\documentclass[11pt,a4paper,reqno]{amsart}
\usepackage{color}
\usepackage{amsmath}
\usepackage{bm}
\usepackage{amsthm}
\usepackage{amsfonts}
\usepackage{amssymb}
\usepackage{enumerate}
\usepackage{color}

\newcommand\R{{\mathbb{R}}}

\newcommand\F{{\mathbb{F}}}
\newcommand\calF{{\mathcal{F}}}
\newcommand\ord{{\mathrm{ord}}}
\newcommand\Fp{{\mathbb{F}_p}}
\newcommand\Fpd{{\mathbb{F}_p^d}}

\newcommand\Z{{\mathbb{Z}}}
\newcommand\N{{\mathbb{N}}}
\newcommand\Q{{\mathbb{Q}}}

\theoremstyle{plain}
  \newtheorem{theorem}{Theorem}

  \newtheorem{proposition}{Proposition}
  
  \newtheorem{lemma}{Lemma}
  \newtheorem{claim}{Claim}

\theoremstyle{remark}

\theoremstyle{definition}

\begin{document}

\title{Additive bases in groups}
\begin{abstract}
In this paper, we study the problem of removing an element from an additive basis in a general abelian group. We introduce analogues 
of the classical functions $X, S$ and $E$ (defined in the case of $\N$) and obtain bounds on them. Our estimates on 
the functions $S_G$ and $E_G$ are valid for general abelian groups $G$ while in the case of $X_G$ we show that distinct types of behaviours
may occur depending on $G$.
\end{abstract}

\author[Lambert, L\^e and Plagne]{Victor Lambert, Th\'ai~Ho\`ang~L\^e and Alain Plagne}
\date{\today}

\address{Centre de Math\'ematiques Laurent Schwartz, \'Ecole polytechnique, 91128 Palaiseau Cedex, France}
\email{victor.lambert@ens-cachan.org}
\email{thai-hoang.le@polytechnique.edu}
\email{plagne@math.polytechnique.fr}
\maketitle

\section{Introduction}
\subsection{Background}
Let $(G,+)$ be an abelian semigroup. If $A \subset G$, then for $h$ a positive integer, $hA$ denotes as usual the $h$-fold sumset of $A$ that is, 
the set of sums of $h$ non-necessarily distinct elements of $A$. For two subsets $A,B$ of $G$, we write $A \sim B$ if the symmetric difference of 
$A$ and $B$ is finite. 

In this paper, we are concerned with the notion of \textit{additive basis}. Several related notions should be defined.

We say $A$ is an \textit{exact asymptotic basis} (from now on, we will simply say a \textit{basis}) of order at most $h$ if all but finitely many elements of $G$ can be expressed 
as a sum of \textit{exactly} $h$ elements of $G$, in other words, if $hA \sim G$.
If $h$ is the smallest integer for which this holds, we say that $A$ is a basis of order $h$ and write $\ord^{*}_G(A)=h$. 
If no such $h$ exists, we write $\ord^{*}_G(A)=\infty$.

We say $A$ is a \textit{weak basis} of order at most $h$ if all but finitely many elements of $G$ can be expressed as a sum of \textit{at most} $h$ elements of $G$, 
in other words, if 
\begin{equation}
\label{weak}
\bigcup_{i=1}^h iA \sim G.
\end{equation}
If $h$ is the smallest integer for which this holds, we say that $A$ is a weak basis of order $h$ and write $\ord_G(A)=h$. If no such $h$ exists, we write $\ord_G(A)=\infty$.

Finally, a basis $A$ of order at most $h$ is called \textit{nice} if $hA=G$. A weak basis $A$ of order at most $h$ is called \textit{nice} if \eqref{weak} 
is in fact an equality. 

These notions are related by the following simple observation. Assume that the ambient semigroup $G$ contains the neutral element $0$. 
Then $A$ is a weak basis if and only if $A \cup \{0\}$ is a basis. Furthermore, $\ord_G(A) = \ord^{*}_G(A \cup \{0\})$. 
(Weak) bases are of interest only when $G$ is infinite. On the other hand, nice (weak) bases make sense in any semigroup.

Historically, additive bases have been studied in the case where $G=\N$, the semigroup of nonnegative integers. In the fundamental paper \cite{eg}, 
Erd\H{o}s and Graham studied the following problem (note that the original formulation of Erd\H{o}s and Graham is slightly different, but equivalent, 
see Section \ref{sec:egcriteria}): 
let $A \subset \N$ be a basis and $a \in A$, when is $A \setminus \{ a \}$ a basis? If $A \setminus \{ a \}$ is still a basis, what can we say about its order? 
Since then, this question and related questions have been extensively studied. 
We give here a brief survey of state-of-the-art results in this theme of research. For more detailed accounts, we refer the reader to \cite{p} or \cite{grekos}.

It turns out that there are only finitely many elements $a \in A$ such that
$A \setminus \{ a \}$ is not a basis, which we refer to as \textit{exceptional} elements. An element $a \in A$ which is not exceptional is called \textit{regular}. Let $A^*$ denote the set of all regular elements of $A$. 
Grekos \cite{grekos2} showed that the number of exceptional elements of $A$ can be bounded in terms of $h$ alone, thus we can define the function
\begin{equation} 
\label{eq:def-e}
 E(h) = \max_{hA \sim \N} |A \setminus A^*|.
\end{equation}
By the third author's work \cite{p3}, we have the following asymptotic formula as $h \rightarrow \infty$:
\begin{equation}
\label{eq:res-e}
E(h) \sim 2 \sqrt{\frac{h}{ \log h}}.
\end{equation}

Erd\H{o}s and Graham \cite{eg} showed that when $a$ is regular, the order of $A \setminus \{ a \}$ can be bounded in terms of $h$ alone. Thus we can define the function
\begin{equation} 
\label{eq:def-x}
X(h) = \max_{hA \sim \N} \max_{a \in A^*} \ord^*(A \setminus \{a\}).
\end{equation}
To date, the best upper and lower bounds for $X(h)$ are both due to the third author \cite{p2}, improving earlier works by St\"ohr \cite{stohr}, Grekos \cite{grekos2} 
and Nash \cite{nash} notably. We have
\begin{equation} \label{eq:res-x}
 \left[ \frac{h(h+4)}{3} \right] \leq X(h) \leq \frac{h(h+1)}{2} + \left\lceil \frac{h-1}{3} \right\rceil .
\end{equation}
It is conjectured by Erd\H{o}s and Graham \cite{eg2} that there is a constant $\alpha$ such that $X(h) \sim \alpha h^2$ as $h \rightarrow \infty$, but this remains open. 
The inequalities in  \eqref{eq:res-x} imply that $X(1)=1, X(2)=4$, $X(3)=7$, but even the value of $X(4)$ remains unknown.

In \cite{grekos2, grekos3}, Grekos observed that in examples of bases $A$ of order $h$ that give large values of $X(h)$, there are actually very few elements $a \in A$ such that $\ord^*(A \setminus \{a\}) = X(h)$. 
This led him to introduce the function
\begin{equation} \label{eq:def-s}
 S(h) = \max_{hA \sim \N} \limsup_{a \in A^*} \ord^*(A \setminus \{a\})
\end{equation}
and to conjecture that the order of magnitude of $S$ is smaller than the one of $X$. This was confirmed by
Cassaigne and the third author \cite{eg} who proved that
\begin{equation} \label{eq:res-s}
h+1 \leq S(h) \leq 2h
\end{equation}
for all $h \geq 2$ (evidently, $S(1)=1$). It was also determined that $S(2)=3$. However, the value of $S(3)$ is still unknown. 
It is an open problem which looks already difficult to determine whether there is a constant $\beta$ such that $S(h) \sim \beta h$ 
as $h \rightarrow \infty$. 

\subsection{Our work}
The goal of this paper is to study the analogues of the functions $E, X, S$ defined above when $G$ is an arbitrary infinite abelian group. 
Clearly, this problem makes sense in any semigroup. 
However, the rich structure of a group gives us more tools and flexibility. As such, our results are not generalizations of results in $\N$, but rather their analogues. 
Indeed, many of the results in $\N$ do not apply automatically to $\Z$, and vice versa.

Before studying the problem of removing elements from a basis in a group $G$, it is quite natural to ask if $G$ has any basis of order $\geq 2$ at all. 
We will show that not only has $G$ a basis, but it also has a 
\textit{minimal basis} of any prescribed order. A basis $A$ of order $h$ of $G$ is called minimal if for any $a \in A$, $A \setminus \{ a \}$ is no longer a basis 
of order $h$ (it could be a basis of some larger order). In other words, any element of $A$ is necessary in order for $A$ to be a basis of order $h$.
This result is the content of our first theorem.

\begin{theorem} 
\label{th:main-minimal}
Let $G$ be any infinite abelian group and $h$ be an integer, $h \geq 2$. Then $G$ has a nice minimal basis of order $h$.
\end{theorem}

We can now talk about the analogues of the functions $E,X,S$ defined on $\N$. Let $A$ be a basis of $G$. 
An element $a \in A$ is called \textit{exceptional} if 
$A \setminus \{ a \}$ is not a basis of $G$ of any order, and \textit{regular} if it is not exceptional. Let $A^{*}$ be the set of regular elements of $A$.
We show that similarly to the case of $\N$, there are only finitely many exceptional elements in $A$. More precisely, define
\[
 E_G(h) = \max_{hA \sim G} |A \setminus A^{*}|
\]
A priori, it is not clear that this function is well-defined (i.e., the number of exceptional elements of $A$ can be bounded in terms of $h$ alone). However, we will prove the following result.

\begin{theorem} 
\label{th:main-e}
\begin{enumerate}[(i)]
\item \label{th:main-ea} For any infinite abelian group $G$ and any integer $h \geq 2$, we have
\[
 E_G(h) \leq h-1.
\]
\item \label{th:main-eb}  There is an infinite group $G$, for which $E_G(h)=h-1$ for any integer $h \geq 2$.
\item \label{th:main-ebsecond} For each integer $h \geq 2$, there is an infinite group $G$ (depending on $h$) for which $E_G(h)=0$.
\end{enumerate}
\end{theorem}

The statements \eqref{th:main-eb} and \eqref{th:main-ebsecond} in this theorem show that, 
as far as general groups are concerned, the upper bound \eqref{th:main-ea} is the best possible one can hope for. Also, it is easy to see that $E_G(1)=0$ for any $G$.

Next we turn to the question of bounding the order of $A\setminus \{a\}$ when $a \in A$ is a regular element. Define
\begin{equation} \label{eq:defx1}
 X_G(h) = \max_{hA \sim G} \max_{a \in A^*} \ord_G^*(A\setminus \{ a \}).
\end{equation}
Studying the function $X_G$ turns out to be less successful than $E_G$. 
We do not even know if $X_G(h)$ is finite for each $G$, 
not to mention the problem of proving that $X_G(h)$ can be bounded in terms of $h$ alone. 
However, in the case of some particular groups we are able to prove bounds for $X_G(h)$. 

In order to state our next result, we shall need the arithmetic function $\Omega$. Recall that it is defined by 
\begin{equation} 
\label{eq:omega}
\Omega(n)=\alpha_1 + \cdots + \alpha_k,
\end{equation}
if $n=p_1^{\alpha_1} \cdots p_k^{\alpha_k}$ is the prime factorization of $n$. 

Another notation that we shall need is for $A$ a subset of a group $G$ and $m$ an integer
 \[
m\cdot A = \{ma : a \in A\}.
\]

By adapting Erd\H{o}s-Graham's argument from \cite{eg}, we prove the following statement.

\begin{theorem} \label{th:main-x1}
Let $G$ be an infinite abelian group. If for any integer $1 \leq m \leq h$, the quotient group $G / m \cdot G$ is finite 
then we have
\begin{equation} 
\label{eq:res-x2}
X_G(h) \leq h^2 + h \cdot \max_{1 \leq m \leq h} \Omega \left( \left| G/m \cdot G \right| \right) +h-1. 
\end{equation}
\end{theorem}

Groups that satisfy the hypothesis of Theorem \ref{th:main-x1} include large classes such as finitely generated groups, 
divisible groups (i.e., groups $G$ such that $m \cdot G = G$ for all $m \in \Z^+$, which include $\R$ and $\Q$) and $\Z_p$ (the $p$-adic integers). 

As for lower bounds, 
the same as in \eqref{eq:res-x} applies to groups which have $\Z$ as a quotient. That is, we can prove the following.

\begin{theorem} \label{prop:x-lower}
Let $G$ be an infinite abelian group.
Suppose there is a subgroup $H$ of $G$ such that $G/H \cong \Z$, then for any integer $h \geq 1$, we have
\[
X_{G}(h)\ge \left[\frac{h(h+4)}{3}\right].
\]
\end{theorem}

This prompts one to believe that the growth of $X_G(h)$ in general is quadratic. 
However, we show that this is not the case by exhibiting another class of groups for which $X_G(h)$ has a \textit{linear} growth.

\begin{theorem} \label{th:main-x2}
Let $p$ be a prime number and $G$ be an infinite abelian group with the property that every nonzero element of $G$ has order $p$.
\begin{enumerate}[(i)]
\item \label{th:main-x2a} For any integer $h\geq p$, we have
\[
 X_G(h) \leq ph + p-1.
\]
\item \label{th:main-x2b} For any integer $h \geq 3(p-1)/2$, we have
\[
 X_G(h) \geq 2h - 3p +3.
\]
\end{enumerate}
In particular, if $p=2$, then $X_G(h) \sim 2h$ as $h \rightarrow \infty$.
\end{theorem}

Though in general we do not know if $X_G(h)$ is finite, we can confirm this when $h=2$ or $h=3$ (for any infinite abelian group $G$). Clearly, $X_G(1)=1$.

\begin{theorem} 
\label{th:main-x3}
For any infinite abelian group $G$, we have
\begin{enumerate}[(i)]
\item \label{th:main-x3a} $3 \leq X_G(2) \leq 5$. 
\item \label{th:main-x3b} $4 \leq X_G(3) \leq 17$.
\end{enumerate}
\end{theorem}

Finally, we turn to the analogue of $S$. We define $S_G(h)$ to be the minimum value of $s$ such that for all $A$ satisfying $hA \sim G$, 
there are only finitely many elements $a \in A$ with the property that 
\[
\ord_G^* (A \setminus \{ a\} ) > s.
\]
Again, a priori, it is not clear that $S_G(h)$ is well defined (though it is clear that $S_G(1)=1$ for any G). It follows immediately from the definition that $S_G(h) \leq X_G(h)$. 
We show that for $S_G(h)$, we have exactly the same bounds as in (\ref{eq:res-s}), 
by generalizing the arguments from \cite{cp}. In doing so, we need to use the notion of \textit{amenability}, which makes the argument no longer elementary. 

\begin{theorem} \label{th:main-s1}
For any infinite abelian group $G$ and $h \geq 2$, we have $h+1 \leq S_G(h) \leq 2h$.
\end{theorem}

In contrast with Theorem \ref{th:main-e} but as in the case of $\N$, we do not know if these bounds are best possible. 
However, we can prove the following equality.

\begin{theorem} 
\label{th:main-s2}
For any infinite abelian group $G$, we have $S_G(2)=3$.
\end{theorem}

The structure of the paper is as follows. In Section \ref{sec:prelim}, we will prove some useful facts, including a generalization 
of the Erd\H{o}s-Graham's criterion for regular elements of a basis.
In Section \ref{sec:minimal}, we will prove Theorem \ref{th:main-minimal}. 
In Section \ref{sec:e}, we will prove results related to the function $E_G$, including Theorem \ref{th:main-e}.
In Section \ref{sec:x}, we will prove results related to the function $X_G$, including Theorems \ref{th:main-x1}, \ref{prop:x-lower}, \ref{th:main-x2} and \ref{th:main-x3}.
Finally, we will prove results related to the function $S_G$, including Theorems \ref{th:main-s1} and \ref{th:main-s2} in Section \ref{sec:s}.

\section{Preliminaries} 
\label{sec:prelim}

\subsection{Some observations} 
\label{sec:obs} 

We first state some simple observations which we will use later on. Some of them are immediate and the proofs will be omitted.

\begin{lemma} \label{lem:translate}
Let $G$ be an infinite abelian group and $A \subset G$. 
If $A$ is a (nice) basis of $G$ and $b \in G$, then $A - b = \{a - b: a \in A \}$ is a (nice) basis of the same order.
\end{lemma}

\begin{proof}
This is immediate since, for any integer $h$,  $h(A-b) = hA -hb$.
\end{proof}

Suppose $H$ is a subgroup of $G$. For $x \in G$, let $\overline{x}$ denote the coset of $x$ in $G/H$ (we distinguish between elements of $G/H$ and subsets of $G$).
Then there is a natural way to produce a basis for $G$.

\begin{lemma} 
\label{lem:rep}
Let $G$ be an abelian group and $H$ be a subgroup of $G$. Let $\Lambda \subset G$ be a system of representatives of $G/H$ in $G$, that is, for any $x\in G$, there is exactly one element $\lambda \in \Lambda$
such that $x+H=\lambda+H$. Then every $x \in G$ can be expressed in a unique way as
\[
x = \lambda + g
\]
where $\lambda \in \Lambda, g \in H$. In particular, if $H \neq G$ and $H \neq \{0\}$ then $\Lambda \cup H$ is a nice basis of order 2 of $G$.
\end{lemma}

In constructing bases for $G$ we will need special systems of representatives which are given by the following

\begin{lemma} 
\label{lem:rep2}
Let $G$ be an abelian group and $H$ be a subgroup of $G$. Then there is a system of representatives $\Lambda$ of $G/H$ in $G$ such that $0 \in \Lambda$ and $\Lambda = -\Lambda$.
\end{lemma}

\begin{proof}
 We select one representative from each coset of $H$ in $G$. Of course, we can select in such a way that $0$ is the representative of $H$, and if $\lambda$ is the representative of a coset $B$, then $-\lambda$ is the representative of the coset $-B$. 
\end{proof}

The next observation says that nice bases can be lifted from quotients to the whole group, a property not satisfied by mere bases.

\begin{lemma} 
\label{lem:lifting}
Let $G$ be an infinite abelian group and $H$ be a subgroup of $G$. Let $A \subset G/H$,
\[
B = \{x \in G : \overline{x} \in A \}
\]
and $h$ be a positive integer. Then we have:
\begin{enumerate}[(i)]
\item $hA = G/H$ if and only if $hB = G$,
\item $\bigcup_{i=1}^h iA = G/H$ if and only if $\bigcup_{i=1}^h iB = G$.
\end{enumerate}
\end{lemma}

The next lemma says that all bases are nice, at the cost of increasing the order.

\begin{lemma} \label{lem:whole}
Let $G$ be an infinite abelian group and $A \subset G$. 
\begin{enumerate}[(i)]
\item If $hA\sim G$, then $(h+1)A=G$,
\item If $\bigcup_{i=1}^h i A \sim G$, then $\bigcup_{i=2}^{h+1} iA = G$.
\end{enumerate}
\end{lemma}

\begin{proof}
Suppose $hA\sim G$. Let $x$ be any element of $G$. Then $x-A$ is infinite, so it must have a non-empty intersection with $hA$. Therefore, $x \in (h+1)A$.
 
Suppose $\bigcup_{i=1}^h i A \sim G$. Let $x$ be any element of $G$. Since $x-A$ is infinite, it must have a non-empty intersection with $rA$ for some $1 \leq r \leq h$. Therefore, 
$$
x \in (r+1)A \subset \bigcup_{i=2}^{h+1} iA.
$$
\end{proof}

In finding bounds for $X_G$, we will need the following fact. If two sumsets of $A$ have a non-empty intersection, then we can find an arbitrarily long sequence of sumsets of $A$ whose intersection is also non-empty.

\begin{lemma} \label{lem:intersection}
Suppose $A \subset G$ and $m,n$ be nonnegative integers. 
If
\[
c \in nA \cap (n+m)A
\]
then for any positive integer $k$, we have
\[
kc \in knA \cap (kn + m)A \cap \cdots \cap (kn + km)A.
\]
\end{lemma}

\subsection{Erd\H{o}s-Graham type criteria} \label{sec:egcriteria}
In \cite{eg}, Erd\H{o}s and Graham proved a criterion for weak bases in $\N$. They show that a weak basis $A$ of $\N$ is a basis if and only if
\begin{equation} \label{eq:eg1}
 \gcd(A-A) =1
\end{equation}
where $A-A= \{ a_1 - a_2: a_1, a_2 \in A\}$.
In turn, this implies a criterion for regular elements of a basis in $\N$. If $A$ is a basis of $\N$, then $a \in A$ is regular if and only if
\begin{equation} \label{eq:eg2}
 \gcd(A \setminus \{a\} - A \setminus \{a\}) =1.
\end{equation}
We will now prove extensions of these criteria in an arbitrary group.

\begin{lemma} 
\label{lem:egcrit1} 
Let $G$ be an infinite abelian group and $A$ be a weak basis of $G$. Then $A$ is a basis if and only if $\langle A-A \rangle$, 
the group generated by $A-A$ in $G$, is equal to $G$.
\end{lemma}

\begin{proof}
Suppose $A$ is a weak basis of order at most $h$, that is,
\[
 G \sim \bigcup_{i=1}^h iA.
\]
Let $H = \langle A-A \rangle$. The image of $a$ in $G/H$ is the same, for any $a \in A$. 
Therefore, for any $s$, the image of $sA$ in $G/H$ consists of a single element. This means that $A$ cannot be a basis unless $G=H$.
 
Conversely, suppose $G=H$. We claim that there is a positive integer $n$ such that 
$$
nA \cap (n+1)A \neq \emptyset.
$$ 
Let $a$ be any element of $A$.  Then $a$ can be expressed as a linear combination
 \[
  a= \sum_{k=1}^t \alpha_k (a_k-b_k)
 \]
where $a_k, b_k \in A$ and $\alpha_k \in \Z^+$ for any index $k$.

Hence the element
\[
 c = a + \sum_{k=1}^t \alpha_k b_k = \sum_{k=1}^t \alpha_k a_k
\]
is in both $nA$ and $(n+1)A$, where $n=\sum_{k=1}^t \alpha_k$. By Lemma \ref{lem:intersection}, we have
\[
 (h-1)c \in \bigcap_{i=0}^{h-1} ((h-1)n+i)A.
\]
For all but finitely many $x \in G$, we have
\[
 x - (h-1)c \in \bigcup_{i=1}^h iA.
\]
It follows that for all but finitely many $x \in G$, we have
\[
 x = x - (h-1)c + (h-1) c \in ((h-1)n+h) A.
\]
Thus $A$ is a basis with order $\leq (h-1)n+h$.
\end{proof}

\begin{lemma} \label{lem:egcrit2} 
Let $G$ be an infinite abelian group and $A$ be a basis of $G$. 
Then $a \in A$ is regular if and only if $\langle A \setminus \{ a \} -A \setminus \{ a \} \rangle = G$.
\end{lemma}

\begin{proof} We want to apply Lemma \ref{lem:egcrit1} right away, but we do not know if $A \setminus \{ a \}$ is a weak basis. Instead, we observe that $B:=A-a$ is also a basis, and contains $0$. Therefore,
$B \setminus \{ 0 \}$ is a weak basis. We have
\begin{eqnarray*}
 A \setminus \{ a \} \textup{ is a basis } &\iff& B \setminus \{0\} \textup{ is a basis } \\
 &\iff& \langle B \setminus \{0\} - B \setminus \{0\} \rangle = G\\
 &\iff& \langle A \setminus \{ a \} -A \setminus \{ a \} \rangle = G.
\end{eqnarray*}
\end{proof}

In \cite{eg}, Erd\H{o}s and Graham gave a slightly different but equivalent definition of the function $X$. 
We will revisit their original definition since we find it convenient to work with both definitions. 
If $G$ is an infinite abelian group, we define
\begin{eqnarray} 
 x_G(h)&=&\max \{ \ord_G^{*}(A) : \cup_{i=1}^h iA \sim G \textup{ and } \ord^*_G(A) < \infty \} \label{eq:defx2}\\
  &=& \max \{ \ord_G^{*}(A) : \cup_{i=1}^h iA \sim G \textup{ and } \langle A-A \rangle = G \}. \nonumber
\end{eqnarray}

\begin{lemma}\label{lem:X=g}
For every infinite abelian group $G$, $X_G=x_G$. 
\end{lemma}

\begin{proof}
Let $h$ be a positive integer, $A$ be any basis of order at most $h$ of $G$ and $a\in A$ be any regular element of $A$. Then $B:=A-a$ is also a basis of order at most $h$ and contains $0$.
Therefore, $B \setminus \{ 0 \}$ is a weak basis of order at most $h$. Furthermore, 
\[
\ord_G^*(B \setminus \{ 0 \}) = \ord_G^*(A\setminus \{a\}).
\] This implies that $X_G(h) \le x_G(h)$.

The other direction is a bit less straightforward. From the definitions of $X$ and $x$, clearly we have $h \leq X_G(h)$ and $h \leq x_G(h)$. 
If $x_G(h)=h$ then necessarily $X_G(h)=h = x_G(h)$, since we already know $X_G(h) \le x_G(h)$.
Thus we may assume that $x_G(h) >h$ (we notice that in view of Theorem \ref{th:main-minimal}, which is yet to be proved, this is always the case if $h \geq 2$).
Let $B$ be any weak basis of order at most $h$ of $G$ satisfying $h < \ord_G^{*}(B) < \infty$. Then $0 \not \in B$ (if not, $\ord_G^{*}(B) = \ord_G(B) \leq h$).
Let $A:=B \cup \{ 0\}$, then $A$ is a basis of order at most $h$ and $0$ is a regular element of $A$, since $A \setminus \{0\} = B$. Furthermore, 
\[
\ord_G^*(A \setminus \{ 0 \}) = \ord_G^*(B).
\] This implies that $x_G(h) \le X_G(h)$.
\end{proof}

\section{Existence of minimal bases} \label{sec:minimal}
In the case of $\N$, it has been known since H\"artter \cite{hartter} that $\N$ has minimal bases of any order (though his proof is non-constructive). A concrete example of a minimal basis of order $h$ in $\N$ is given by
\[
 A = \left\{ \sum_{f \in \calF} 2^{f} : \calF \textup{ is a finite set of distinct nonnegative integers congruent modulo }h\right\}.
\]
See \cite{nathanson} for further quantitative properties of this basis.

We first show that in a general group $G$, there are nice bases of any order as long as there is a special representation of elements of $G$ similar to base 2 representation.

\begin{proposition} \label{prop:digits}
 Let $G$ be an infinite abelian group. Suppose that there is an infinite sequence of subsets $(\Lambda_i)_{i=0}^\infty$ of $G$ satisfying the following properties:
 \begin{enumerate}[(i)]
  \item $0 \in \Lambda_i$, for any $i \in \N$,
  \item $-\Lambda_i = \Lambda_i$, for any $i \in \N$,
  \item Every element $x \in G$ has a \textup{unique} representation as
  \[
   x = \lambda_0(x) + \lambda_1(x) + \cdots
  \]
where $\lambda_i(x) \in \Lambda_i$ for any $i$, and $\lambda_i(x)=0$ for all but finitely many indices $i$. 
In other words, $G$ is equal to the ``direct sum'' $\oplus_{i=0}^{\infty} \Lambda_i$.
\footnote{Strictly speaking, we cannot talk about direct sums here since the $\Lambda_i$ are merely sets, not groups.} 
 \end{enumerate}
Then for any integer $h \geq 2$, $G$ has a nice minimal basis of order $h$.
\end{proposition}

\begin{proof}
For $x \in G$, we refer to the set $\{ i \in \N : \lambda_i(x) \neq 0\}$ as the \textit{support} of $x$. 

Clearly, if $x$ and $y$ have disjoint supports, then 
\begin{equation*} \label{eq:homo}
 \lambda_{i}(x+y) = \lambda_{i}(x) + \lambda_{i}(y).
\end{equation*}

Let $\N = N_1 \cup \cdots \cup N_h$ be a partition of $\N$ into $h$ infinite disjoint sets. Let $A_j$ be the set of all $x \in G$ supported on $N_j$. Put
\[
 A = \cup_{j=1}^h A_j.
\]
By definition, $0 \in A$. Clearly, any element $x \in G$ can be expressed in a \textit{unique} way as
\begin{equation} \label{eq:repx}
 x = a_1 + \cdots +a_h
\end{equation}
where $a_j \in A_j$ for any $j=1, \ldots, h$. When $a_1, \ldots, a_h \neq 0$, $x$ cannot be written as a sum of fewer than $h$ elements from $A$. 
This shows that $A$ is a basis of order $h$. However, $A$ is not minimal. We claim that $B:=A \setminus \{0\}$ is a nice, minimal basis of order $h$. 

First we show that $hB=G$. In the expression (\ref{eq:repx}), some (or even all) of the $a_j$ can be 0. We now observe that any (zero or non-zero) element in $A_j$ can be expressed as a sum of two \textit{non-zero} elements of $A_j$.
Indeed, if $a \in A_j$, then $a$ can be written as
\[
 a = (a + \lambda) + (-\lambda)
\]
where $\lambda$ is any element in $\Lambda_k \setminus \{0\}$ and $k \in N_j$ is any element not in the support of $a$. Note that by hypothesis, $-\lambda \in \Lambda_k$ as well. 
Thus starting from (\ref{eq:repx}) we can increase the number of non-zero elements by one at a time, which shows that $hB=G$.

It remains to see that $B$ is a minimal basis. Let $a$ be any element in $B$.  Without loss of generality, we may assume $a \in A_1 \setminus \{ 0\}$.
Consider an element $x \in G$ of the form
\[
x= a + a_2 +\cdots + a_h
\]
where $a_j \in A_j \setminus \{ 0 \}$ for any $j=2, \ldots, h$. Then there is a unique way to write $x$ as a sum of $h$ elements of $B$, and $a$ appears in this expression. 
Therefore, $x$ cannot be written as a sum of $h$ elements from $B \setminus \{a\}$. Since there are infinitely many elements $x$ of this form, it follows that  $\ord^* (A \setminus \{ a \} ) \geq h+1$.
\end{proof}

The proof of Theorem \ref{th:main-minimal} now follows.

\begin{proof}[Proof of Theorem \ref{th:main-minimal}]
It suffices to construct a sequence $(\Lambda_i)_{i=0}^\infty$ satisfying the hypothesis of Proposition \ref{prop:digits}. We distinguish two cases.

\noindent \textit{Case 1:} $G$ has an element of infinite order. We may assume that $\Z < G$. 
Let $\Lambda_0$ be a system of representatives of $G/ \Z $ in $G$. By Lemma \ref{lem:rep}, any element $x \in G$ can be written in a unique way as
 \[
  x = n + \lambda_0
 \]
where $\lambda_0 \in \Lambda_0$ and $n \in \Z$. Furthermore, by Lemma \ref{lem:rep2}, we may choose $\Lambda_0$ in such a way that $0 \in \Lambda_0$ and $\Lambda_0 = - \Lambda_0$.
Observe that every integer $n$ can be written in a \textit{unique} way as
\[
 n = \sum_{i=0}^{k} a_{i} 3^{i}
\]
where $a_i \in \{0, 1, -1\}$ for any $i$ (this is known in the literature as the \textit{balanced ternary} representation of $n$). Put $\Lambda_{i} = \{0, 3^{i-1}, -3^{i-1} \}$. 
Then any element $x \in G$ can be written in a unique way as
\begin{equation} \label{eq:directsum}
  x = \lambda_0(x) + \lambda_1(x) + \cdots
\end{equation}
where $\lambda_i(x) \in \Lambda_i$ for any $i$, and $\lambda_i(x)=0$ for all but finitely many indices $i$.

\noindent \textit{Case 2:} Every element of $G$ has finite order.

Let $g_1 \in G$ be any element. Then $G_1:=\langle g_1 \rangle$ is finite. We can find $g_2 \in G \setminus G_1$. Put $G_2 := \langle g_1, g_2 \rangle$, then $G_1 \lneq G_2$ and $G_2$ is finite. 
This way, we have an infinite chain of subgroups of $G$
\[
 G_1 \lneq G_2 \lneq \cdots.
\]
For each integer $i \geq 2$, let $\Lambda_i \ni 0$ be a system of representatives of $G_{i} / G_{i-1}$ in $G_i$. By Lemma \ref{lem:rep}, any $x \in G_{i}$ can be written in a unique way as 
\[
 x = \lambda + g
\]
where $\lambda \in \Lambda_{i}$ and $g \in G_{i-1}$. We also put $\Lambda_1 = G_1$. Thus every $x \in \cup_{i=1}^\infty G_i$ can be written in a unique way as
\[
 x = \lambda_1(x) + \lambda_2(x) + \cdots
\]
where $\lambda_i \in \Lambda_i$ for any $i=1, 2, \ldots$, and all but finitely many $\lambda_i$ are zero (indeed, if $x \in G_k$, then $\lambda_i(x)=0$ for all $i \geq k+1$). 

Finally, let $\Lambda_0 \ni 0$ be a system of representatives of $G / \cup_{i=1}^\infty G_i$ in $G$. Then every $x$ in $G$ can be written in a unique way as
\[
 x = \lambda_0(x) + \lambda_1(x) + \lambda_2(x) + \cdots
\]
where $\lambda_i \in \Lambda_i$ for any $i=0, 1, 2, \ldots$, and all but finitely many $\lambda_i$ are zero. 
Furthermore, by Lemma \ref{lem:rep2}, we may require that $\Lambda_i = -\Lambda_i$ for $i=0$ and any $i \geq 2$ (this is certainly satisfied when $i=1$). 
\end{proof}

\section{The function $E_G$} \label{sec:e}
In this section, we study bounds for $E_G$.

\begin{proof}[Proof of Theorem \ref{th:main-e} \eqref{th:main-ea}] We will show that if $hA \sim G$, then $A$ cannot have more than $h-1$ exceptional elements.

By Lemma \ref{lem:egcrit2}, if $a$ is an exceptional element, then $\langle A-A \rangle = G$ but $\langle A \setminus \{ a\} - A \setminus \{ a\} \rangle \neq G$. 
This implies that $a-a'$ is not in $\langle A \setminus \{ a\} - A \setminus \{ a\} \rangle$ for \textit{some}
(and hence for \textit{all}) $a' \in A \setminus \{a\}$.

Suppose there are at least $h$ exceptional elements $a_1, \ldots, a_{h}$ in $A$. Since $G$ is infinite, so is $A$. Let $a_0$ be an element in $A \setminus \{a_1, \ldots, a_h \}$. 
Since $hA \sim G$, we can find $a \in A \setminus \{a_0, a_1, \ldots, a_{h} \}$ such that
the element 
\[
 a_0+a_1+a_2+\cdots+a_{h} - a
\]
can be expressed as a sum $b_1+\cdots+b_h$ of $h$ elements in $A$. Therefore,
\[
 \sum_{i=0}^{h} (a_i-a) = \sum_{i=1}^h (b_i - a).
\]
Some of the $b_i$ may be equal to some of the $a_i$. We have two possibilities. \\

\noindent \textit{Case 1:} $\{a_1, a_2, \ldots, a_h\} \neq \{b_1, b_2, \ldots, b_h\}$.
This means that after canceling common terms, some $a_i$ (where $i \neq 0$) must remain on the left hand side. 
But this implies that $a_i-a \in \langle A \setminus \{ a_i\} - A \setminus \{ a_i\} \rangle$, a contradiction.  

\noindent \textit{Case 2:} $\{a_1, a_2, \ldots, a_h\} = \{b_1, b_2, \ldots, b_h\}$.
This implies that $a_0 = a$, a contradiction.
\end{proof}

A remark should be made here. In $\N$, the fact that any basis has only finitely many exceptional elements follows immediately 
from Erd\H{o}s-Graham's criterion \eqref{eq:eg1} (see \cite[Teorema 1]{p}). However, that proof relies on a special property of $\Z$, 
namely that all strictly increasing sequences of subgroups of $\Z$ are finite. As such, it cannot be generalized to general groups.

Theorem \ref{th:main-e} \eqref{th:main-eb} and \eqref{th:main-ebsecond} follow immediately from the following proposition.

\begin{proposition} \label{prop:efpt} 
Let $G=\Fp[t]$ be the ring of polynomials over a prime field $\F_p$. For any integer $h\geq 2$, we have
$$
E_{G}(h) = \left[ \frac{h-1}{p-1} \right].
$$
\end{proposition}

In particular, if $p=2$ then $E_G(h)=h-1$ for all $h \geq 2$. 
On the other hand, there is no non-trivial universal lower bound for $E_G(h)$, since $E_{G}(h) = 0$ when $p>h$.

\begin{proof} 
First we show that  
$$
E_{G}(h) \leq \left[ \frac{h-1}{p-1} \right] .
$$ 
We argue similarly to the proof of Theorem \ref{th:main-e} \eqref{th:main-ea}.

Suppose $A \subset \Fp[t]$ is a basis of order $h$, and $a_1, \ldots, a_k$ are all the exceptional elements of $A$. Suppose for a contradiction that $k (p-1) \geq h$. Then there exists $0 \leq \alpha_1, \ldots, \alpha_k \leq p-1$ such that
\[
 \alpha_1+\cdots+\alpha_k = h.
\]
Let $a_0$ be another element in $A \setminus \{a_1, \ldots, a_k\}$.  Since $hA \sim G$ and $A$ is infinite, there is $a \in A \setminus \{a_0, a_1, \ldots, a_k \}$ such that the element 
\[
\sum_{i=1}^k \alpha_i a_i  + a_0 - a
\]
can be expressed as a sum $\sum_{j=1}^h b_i$  of $h$ elements of $A$. Therefore,
\[
\sum_{i=1}^k \alpha_i (a_i-a)  + (a_0 - a) = \sum_{j=1}^h (b_i-a).
\]
Since $a_0 - a \neq 0$, the multisets $\{ a_1 (\alpha_1 \textrm{ times}), \ldots, a_k (\alpha_k \textrm{ times})\}$ and $\{b_1, \ldots, b_h\}$ are distinct. Therefore, after canceling common terms, 
there is $1 \leq i \leq k$ and some $0 < \beta \leq \alpha_i$ such that $\beta(a_i-a)$ lies in $\langle A\setminus \{a_i \} - A\setminus \{a_i \} \rangle$.
This in turn implies that $a_i-a$ lies in this subspace as well (here we are using the fact that $\Fp$ is a field!), which contradicts Lemma \ref{lem:egcrit2} since $a_i$ is exceptional.

Therefore, $h-1 \geq (p-1)k$ and consequently $k \leq [(h-1)/(p-1)]$.

The following simple example shows that equality is attained. Let 
$$
k=\left[ \frac{h-1}{p-1} \right] .
$$ 
Perform the Euclidean division $h=k(p-1)+r+1$ where $0\le r <p-1$. 

Let 
\[
A=\left\{1, t, \cdots, t^{k-1}\right\} \bigcup t^{k} \cdot\Fp \bigcup\cdots\bigcup t^{k+r-1}\cdot\Fp \bigcup t^{k+r}\cdot\Fp[t].
\]
(The sets $t^{k} \cdot\Fp, \ldots, t^{k+r-1}\cdot\Fp$ are not there if $r=0$.) 

Then $A$ is a basis of order $k(p-1)+r+1=h$. Indeed, it is easy to see that all elements in $\Fp[t]$ can be expressed as a sum of $k(p-1)+r+1$ elements from $A$ (note that $0 \in A$).
Furthermore, for all $P(t) \in \F[t] \setminus \{ 0 \}$, the element 
\[
\sum_{i=0}^{k-1} (p-1)t^i + \sum_{i=k}^{k+r-1}t^i + P(t)t^{k+r}
\]
cannot be expressed as a sum of fewer than $h$ elements from $A$.

Using Lemma \ref{lem:egcrit2}, it is easy to see that the exceptional elements in $A$ are exactly 
\[
\left\{ t^i: i=0, \ldots, k-1 \right\}.
\]
\end{proof}

\section{The function $X_G$} 
\label{sec:x}
In this section, we study bounds for $X_G$. We remind the reader that we will use freely both definitions for $X_G$, namely (\ref{eq:defx1}) and (\ref{eq:defx2})
which coincide by Lemma \ref{lem:X=g}.

\subsection{General bounds}
In proving Theorem \ref{th:main-x1}, we will need the following (recall that the function $\Omega$ is defined in \eqref{eq:omega}).

\begin{lemma} 
\label{lem:torsion}
Let $G$ be a finite abelian group which is $m$-torsion (that is, $mx=0$ for all $x \in G$).  
Let $A \subset G$ satisfy $\langle A - A\rangle =G$. Then for any integer $s \geq \Omega(|G|)$, we have $smA=G$. 
\end{lemma}

\begin{proof}
Since $\langle A - A\rangle =G$, we can choose elements $a_1, a_2, \ldots $ of $A$ in such a way that for any integer $k$, if $\langle A_k - A_k \rangle \neq G$, 
then $\langle A_k - A_k \rangle \lneq \langle A_{k+1} - A_{k+1} \rangle$, where
\[
 A_k = \{a_1, \ldots, a_k\}.
\]
It is easy to see that any strictly increasing sequence of subgroups of $G$ has length at most $\Omega(|G|)$. 
Hence for some integer $t \leq \Omega(|G|)$, we have $\langle A_t - A_t \rangle = G$. Thus every element $x \in G$ has a representation
\[
 x = \sum_{i, j=1}^t \alpha_{i,j} (a_i - a_j)
\]
where $\alpha_{i,j} \in \Z$. By rearranging the right-hand side, and since $G$ is $m$-torsion, this implies that we have a representation
\[
 x = \sum_{i}^t \beta_{i} a_i 
\]
where $0 \leq \beta_i < m$ for any $i=1, \ldots, t$ and $\sum_{i=1}^{t} \beta_i$ is a multiple of $m$. Since $0 \in mA$, we can add as many zeroes as we want and have
$x \in tm A \subset smA$, as desired.
\end{proof}

\begin{proof}[Proof of Theorem \ref{th:main-x1}] 
We use the definition (\ref{eq:defx2}). Let $A$ be a weak basis of order at most $h$ of $G$ satisfying $\langle A-A\rangle=G$. 
Let 
$$
s=\max_{1 \leq m \leq h} \Omega \left( \left| G/ m \cdot G\right| \right).
$$

Since $(h+1) A \subset G \sim \bigcup_{i=1}^h iA$, there must be some integer $n$ satisfying $1 \leq n \leq h$ such that $nA \cap (h+1)A \neq \emptyset$. 
 
Let $m = h+1-n$ and $c \in nA \cap (n+m)A$. By Lemma \ref{lem:intersection}, we have
\[
  (h-1)c \in \bigcap_{i=0}^{h-1} ((h-1)n+im)A.
\] 
Since $G \setminus \left( \bigcup_{i=1}^h i A \right)$ is finite, we conclude that
\[
  m \cdot G \setminus \left( \bigcup_{i=1}^h mi A \right) 
\] 
is also finite. It follows that
\begin{equation} \label{eq:basis1}
(h-1)c + m \cdot G \setminus ( (h-1)n + hm ) A 
\end{equation}
is finite.

On the other hand, the group $G / m \cdot G$ is finite and clearly $m$-torsion. Also, $\langle \overline{A} - \overline{A} \rangle = G / m \cdot G$, where $\overline{A}$ is the image of $A$ under the projection 
$G \rightarrow G / m \cdot G$. By Lemma \ref{lem:torsion}, we have
\[
 sm \overline{A} = G / m \cdot G.
\]
In other words, there is a system of representatives $\{x_1, \ldots, x_k\}$ of $G / m \cdot G$ in $G$ such that 
\begin{equation} \label{eq:basis2}
x_j \in sm A
\end{equation} 
for any $j=1, \ldots, k$.

For any $x \in G$, there exists $1 \leq j \leq k$ such that $x-(h-1)c - x_j \in m \cdot G$. It follows from (\ref{eq:basis1}) that for 
all but finitely many $x \in G$, we have
\begin{equation} \label{eq:basis3}
x - x_j \in ( (h-1)n + hm ) A.
\end{equation}
By writing
\[
x= x - x_j + x_j
\]
and using (\ref{eq:basis2}) and (\ref{eq:basis3}), we have
\[
 G \sim (sm + (h-1)n + hm)A.
\]
Therefore, $A$ is a basis of order at most
\[
 sm + (h-1)n + hm = sm + (h-1)(m+n) + m \leq h^2 + sh + h-1.
\]
(Recall that $m+n=h+1$.)
\end{proof}

Another remark is worth making here. 
The hypothesis of Theorem \ref{th:main-x1} is satisfied if $G/m \cdot G$ is finite for any $m$. Divisible groups 
(i.e. such that $m\cdot G=G$, for all $m\ge 1$), which include $\R$ and $\Q$, satisfy of course this property. It is easy to see that finitely generated abelian groups also satisfies this property. The group $\Z_p$ of $p$-adic integers also satisfies this property, since $\Z_p / m \cdot \Z_p \cong \Z / p^l \cdot \Z$, where $p^l$ is the highest power of $p$ in $m$. 
 Infinite groups, all of whose proper quotients are finite, (called \textit{just infinite} groups) satisfy this property. Note that $\Z_p$ is not just infinite, since $\Z_p/\Z$ is infinite.

We now turn to lower bounds and prove Theorem \ref{prop:x-lower}. In fact, we are able to ``lift'' the lower bound in \eqref{eq:res-x} (applied to $\N$) 
to more general groups simply because the basis giving this example in $\N$ is in fact a nice basis.

\begin{proof}[Proof of Theorem \ref{prop:x-lower}]
Let 
$$
g=\left[\frac{h(h+4)}{3}\right]+1
$$ 
and $k=g-1$. 
According to Theorem 20 in \cite{p2}, there exists a set $A \subset \Z / g\Z$ of two elements such that :
\begin{enumerate}[(i)]
\item $A\cup 2A \cup \cdots \cup hA=\Z/g\Z$,
\item $(k-1)A\neq \Z/g\Z$,
\item $kA=\Z/g\Z$.
\end{enumerate}
Since $\Z$ is a quotient of $G$, $\Z/ g \Z$ is also a quotient of $G$. That is, there is a subgroup $K$ of $G$ such that $G/K \cong \Z/ g \Z$.

Let $B = \{x \in G: \overline{x} \in A\}$, where $\overline{x}$ denotes the coset of $x$ in $G/K$. 
Then Lemma \ref{lem:lifting} implies that
\begin{enumerate}[(i)]
\item $B\cup 2B \cup \cdots \cup hB=G$,
\item $(k-1)B\neq G$,
\item $kB=G$.
\end{enumerate} 
In other words, $\ord^*_G(B)=k$. By (\ref{eq:defx2}), this implies that 
$$
X_{G}(h)\ge k = \left[\frac{h(h+4)}{3}\right].
$$
\end{proof}

\subsection{The torsion case} In this section, we suppose that $px=0$ for any $x\in G$, where $p$ is a prime. When $G$ is torsion, we can shorten the length of the sequence of sumsets in question, which explains the dramatically improved upper bound for $X_G(h)$.

\begin{proof}[Proof of Theorem \ref{th:main-x2} \eqref{th:main-x2a}] 
Again, we use the definition (\ref{eq:defx2}) of $X_G$. Let $A$ be any weak basis of order at most $h$ and suppose $\ord_G^*(A)=k$. Since $px=0$ for any $x\in G$, we have the inclusion $nA \subset (n+p)A$ for any $n$. Therefore $\cup_{i=h-p+1}^{h} A \sim G$. Lemma \ref{lem:whole} implies that
\begin{equation} \label{eq:union}
 \bigcup_{i=h-p+2}^{h+1} iA = G.
\end{equation}
Clearly, we also have
\begin{equation} \label{eq:union3}
 \bigcup_{i=h-p+3}^{h+2} iA = G.
\end{equation}
We now distinguish two cases.\\

\noindent \textit{Case 1:} $(h+2)A \cap nA =\emptyset$, for any $h-p+3 \leq n \leq h+1$. Then from (\ref{eq:union}) and (\ref{eq:union3}), and since $(h-p+2)A \subset (h+2)A$, we have necessarily
\[
(h-p+2)A = (h+2)A.
\]
By repeatedly adding $pA$ to both sides, we have
\[
(h-p+2)A = (h+2 + lp)A
\]
for any $l \geq 0$. If $l$ is sufficiently large then $h+2 + lp \geq k$ and $(h+2 + lp)A=G$. Therefore, $(h-p+2)A=G$ and $k \leq h-p+2 \leq h$.  

\noindent \textit{Case 2:} $nA \cap (h+2)A \neq \emptyset$ for some $h-p+3 \leq n \leq h+1$. Put $m=h+2-n$. 
We argue as in the beginning of the proof of Theorem \ref{th:main-x1}. If $c \in nA \cap (n+m)A$, then by Lemma \ref{lem:intersection}, we have
\begin{equation} \label{eq:intersection}
  (p-1)c \in \bigcap_{i=0}^{p-1} ((p-1)n+im)A.
\end{equation}
Next we claim that
\begin{equation*}
\bigcup_{i=h-p+1}^h iA \subset \bigcup_{i=0}^{p-1} (h-p+1 + im) A.
\end{equation*}
Indeed, since $(m,p)=1$, $\{im\}_{i=0}^{p-1}$ forms a complete residue system modulo $p$. If $j$ is the least nonnegative residue of $im$ modulo $p$, then $im \equiv j \pmod{p}$ and $im \geq j$, so that 
\[
 (h-p+1 + im) A \supset (h-p+1 +j)A.
\]
Therefore, 
\begin{equation*}
G \sim \bigcup_{i=0}^{p-1} (h-p+1 + im) A.
\end{equation*}
For all but finitely many $x \in G$, we have
\begin{equation} \label{eq:union2}
x - (p-1)c \in \bigcup_{i=0}^{p-1} (h-p+1 + im) A.
\end{equation}
Combining (\ref{eq:union2}) and (\ref{eq:intersection}) we see that for all but finitely many $x\in G$
\begin{align*}
 x \in & \left( (h-p+1) + (p-1)m + (p-1)n \right)A \\
 & = (h-p+1 + (p-1)(h+2)) A = (hp+p-1)A.
\end{align*}
Therefore, $\ord^*_G(A) \leq hp+p-1$.
\end{proof}
In order to find a lower bound for $X_G(h)$ we use the same idea as in Theorem \ref{prop:x-lower}, namely, to find a nice basis in a quotient of $G$. Note that $G$ is an infinite vector space over $\Fp$. 
Consequently, all finite quotients of $G$ are isomorphic to $\Fpd$, for some $d$. 
Nice weak bases of cardinality $d$ in $\Fpd$ are very well understood by the following:

\begin{lemma}
\label{lem:vs1}
Let $A = \{e_1, \ldots, e_d \} \subset \Fpd$. Then $A$ is a nice weak basis of $\Fpd$ if and only if $e_1, \ldots, e_d$ are linearly independent. If this condition is satisfied, then every element in $\Fpd$ can be expressed as a sum of $\leq (p-1)d$
elements from $A$, and $(p-1)d$ is best possible.
\end{lemma}

\begin{proof}
Clearly $iA \in \langle A \rangle$ for any $i$. If $A$ is a nice weak basis, then necessarily $\langle A \rangle=\Fpd$ and consequently $e_1, \ldots, e_d$ are linearly independent. Suppose that $e_1, \ldots, e_d$ are linearly independent. For any $0 \leq \alpha_1, \ldots, \alpha_d \leq p-1$, the element $\sum_{i=1}^d \alpha_i e_i$ is a sum of $\sum_{i=1}^d \alpha_i \leq (p-1)d$ elements from $A$. Furthermore,
$\sum_{i=1}^d (p-1) e_i$ cannot be expressed as a sum of fewer than $(p-1)d$ elements from $A$.
\end{proof}
This leads us to the following characterization of nice bases of cardinality $d+1$ in $\Fpd$.
\begin{lemma} \label{lem:vs2}
Let $A=\left\{e_1,\cdots,e_d,\alpha_1e_1+\cdots+\alpha_d e_d\right\} \subset \Fpd$, where $e_1, \ldots, e_d$ are linearly independent.
Then $A$ is nice basis of $\Fpd$ if and only if 
\[
\sum_{i=1}^d\alpha_i\not\equiv 1 \bmod p. 
\]
If this condition is satisfied, then $d(p-1)A = \Fpd$ and $(d(p-1)-1)A \neq \Fpd$.
\end{lemma}
\begin{proof} We make the simple yet crucial observation that $A$ is a nice basis if and only if $A-a$ is a nice basis (Lemma \ref{lem:translate}). (Note that this property fails for \textit{vector space bases}.) We have 
\[
A-e_1=\left\{0,e_2-e_1,\cdots,e_d-e_1,(\alpha_1-1)e_1+\cdots+\alpha_de_d\right\}.
\]
Clearly, $(A-e_1)$ is a nice basis if and only if $(A-e_1)\setminus\left\{0\right\}$ is a nice weak basis. By Lemma \ref{lem:vs1}, we only need to check when $(A-e_1)\setminus\left\{0\right\}$ is a family of $d$ independent vectors. In the vector space basis $\{e_1, \ldots, e_d \}$, we have 
\begin{eqnarray*}
\det(\left\{e_2-e_1,\cdots,e_d-e_1,(\alpha_1-1)e_1+\cdots+\alpha_de_d\right\})&&\\
 & \hspace*{-4cm} =&  \hspace*{-2cm} 
\left|
\begin{array}{rrrrrrc}
-1& 	-1& 	\cdots 	&\cdots 	& -1	&-1 	&\alpha_1-1\\
 1& 	0 & 	\cdots	&\cdots 	& 0	&0 	&\alpha_2\\ 
 0& 	1 & 	0 	&	\cdots  &0&0 & \alpha_3\\ 
 0&  0 &  1 &\ddots &\cdots &  \vdots & \vdots\\
\vdots& \vdots  & \ddots &\ddots&\ddots& \vdots &\vdots \\ 
\vdots &  \vdots          &            & \ddots & 1 & 0& \vdots\\
0& 0& \cdots&\cdots&0& 1&\alpha_d
\end{array}\right| \\
& \hspace*{-4cm}= &  \hspace*{-2cm}  \sum_{i=1}^d\alpha_i-1,
\end{eqnarray*}
and the first part of Lemma \ref{lem:vs2} follows. The second part of Lemma \ref{lem:vs2} follows from the second part of Lemma \ref{lem:vs1}.
\end{proof}
We can now construct a nice basis in $\Fpd$ which plays a similar role to the set $A$ used in the proof of Theorem \ref{prop:x-lower}.

\begin{lemma}\label{lem:vs3}
Let $e_1,\cdots,e_d$ be $d$ linearly independent vectors in $\Fpd$. Suppose that $d\not\equiv 1\bmod p$.
Then the set
\[
A=\left\{e_1,\cdots,e_d,e_1+\cdots+e_d\right\}
\]
satisfies the following properties:
\begin{enumerate}[(i)]
\item any element in $\Fpd$ can be expressed as a sum of at most $(d+1)(p-1)/2$ elements from $A$,
\item $(d(p-1)-1)A \neq \Fpd$,
\item $d(p-1)A = \Fpd$.
\end{enumerate}
\end{lemma}

\begin{proof}
The last two assertions follow directly from Lemma \ref{lem:vs2} and the assumption that $d\not\equiv 1\bmod p$. 

As for the first one, put $a = \sum_{i=1}^d e_i$.
Consider an arbitrary element $x=x_1 e_1+\cdots+x_d e_d \in \Fpd$. Define 
\[
\alpha_i=\left|\left\{x_j\equiv i\bmod p\right\}\right|.
\]
For all $0\le i\le p-1$, we can write 
\[
x=ia+\sum_{j=1}^d(x_j-i)e_j=ia+\sum_{j=1}^dy_je_i
\]
with $0\le y_j\le p-1$. In this decomposition of $x$, we use $i+\alpha_{i+1}+2\alpha_{i+2}+\cdots+(p-1)\alpha_{i+p-1}$ elements of $A$. Thus, $x$ can be written using $$\min_i\left\{i+\alpha_{i+1}+2\alpha_{i+2}+\cdots+(p-1)\alpha_{i+p-1}\right\}$$ elements of $A$.

Since $$\begin{aligned}\sum_{i=0}^{p-1}\left(i+\alpha_{i+1}+2\alpha_{i+2}+\cdots+(p-1)\alpha_{i+p-1}\right)&=\left(\sum_{i=0}^{p-1}\alpha_i\right)\frac{p(p-1)}{2}+\frac{p(p-1)}{2}\\
&=\left(d+1\right)\frac{p(p-1)}{2}\end{aligned}$$
the minimum among the $i$'s is at most $(d+1)(p-1)/2$, which proves the first assertion of Lemma \ref{lem:vs3}.
\end{proof}

\begin{proof}[Proof of Theorem \ref{th:main-x2} (\ref{th:main-x2b})]
If 
$\left[2h/(p-1)-1\right]\not\equiv 1\bmod p$, 
let 
$$
d=\left[\frac{2h}{p-1}-1\right].
$$ 
If not, we choose 
$$
d=\left[\frac{2h}{p-1}-2\right].
$$ 
Since $h \geq 3(p-1)/2$, we have $d \geq 1$. 

We now proceed as in the proof of Theorem \ref{prop:x-lower}. 
Let $A \subset \Fpd$ be the set given by Lemma \ref{lem:vs3}. Since $(d+1)(p-1)/2 \leq h$, we have
\begin{enumerate}[(i)]
\item $A\cup 2A \cup \cdots \cup hA=\Fpd$,
\item $(d(p-1)-1)A \neq \Fpd$,
\item $d(p-1)A = \Fpd$.
\end{enumerate}
There is a subgroup $K$ of $G$ such that $G/K \cong \Fpd$. Let $B = \{x \in G: \overline{x} \in A\}$, where $\overline{x}$ denotes the coset of $x$ in $G/K$. 
Lemma \ref{lem:lifting} implies that
\begin{enumerate}[(i)]
\item $B\cup 2B \cup \cdots \cup hB=G$,
\item $(d(p-1)-1)B\neq G$,
\item $d(p-1)B=G$.
\end{enumerate} 
This implies that 
$$
X_G(h) \geq \ord_G^{*}(B) = d(p-1) \geq (p-1) \left( \frac{2h}{p-1}-3 \right) = 2h -3p+3.
$$
\end{proof}

We notice that the proof shows that we have a better bound $X_G(h) \geq 2h-2p+2$ in the case $d=\left[2h/(p-1)-1\right]\not\equiv 1\bmod p$. 
Also, if $p=2$, then Theorem \ref{th:main-x2} (\ref{th:main-x2a}) 
 and the proof of Theorem \ref{th:main-x2} (\ref{th:main-x2b}) imply the rather tight bounds, namely 
 \[
  2h-2 \leq X_{\F_2[t]}(h) \leq 2h+1.
 \]
It is perhaps of interest to determine the exact value of $X_{G}(h)$ when $p=2$.

\subsection{When $h$ is small.} In this section we prove Theorem \ref{th:main-x3}. Note that the lower bounds $X_G(2) \geq 3$ and $X_G(3) \geq 4$ are immediate consequences of Theorem \ref{th:main-minimal}. 
In proving the upper bounds, we again use the definition (\ref{eq:defx2}) of $X_G$.

\begin{proof}[Proof of Theorem \ref{th:main-x3}(\ref{th:main-x3a})]
Suppose $$A\cup2A\sim G$$ and $\ord_G^*(A)=k$ is finite. By Lemma \ref{lem:whole}, for every $l\ge2$, we have $$lA\cup(l+1)A=G.$$

\noindent \textit{Case 1:} There exists $c$ in $2A\cap3A$. For all but finitely many $x\in G$, we have $x-c\in A\cup2A$. Thus for all but but finitely many $x \in G$, $x=x-c+c \in 4A$, and $4A \sim G$.

\noindent \textit{Case 2:}  $2A\cap3A=\emptyset$. If there exists $c\in 3A\cap4A$, then by the same argument as above, we have $5A\sim G$. 
Let us assume that $3A\cap4A=\emptyset$. Since $2A\cup3A=3A\cup4A=G$, we deduce $2A=4A$. It follows that $2A=2mA$ for all $m\ge1$. If $2m>k$ then $2mA=G$ and $2A=G$.

In any case we have $\ord_G^*(A) \leq 5$. 
\end{proof}

\begin{proof}[Proof of Theorem \ref{th:main-x3}(\ref{th:main-x3b})]
Suppose $$A\cup2A \cup 3A \sim G$$ and $\ord_G^*(A)=k$ is finite. By Lemma \ref{lem:whole}, for every $l\ge2$, we have $$lA\cup(l+1)A\cup(l+2)A=G.$$

Observe that if $iA\cap(i+1)A\neq\emptyset$, then $2iA\cap(2i+1)A\cap(2i+2)A\neq\emptyset$ and this implies $(2i+3)A\sim G$. Thus we can assume that $iA\cap(i+1)A=\emptyset$ for $1\le i\le 7$. Otherwise, $\ord^*(A)\le 2i+3\le17$.
We distinguish three cases.

\noindent \textit{Case 1:} $0\in 2A$. Then $4A\supset2A$ and $5A\supset3A$. It follows that $3A\cup4A=G=4A\cup5A$. By assumption, these are partitions of $G$. Therefore, $3A=5A$, which implies $3A=(2m+3)A$ for any $m \geq 1$.
Consequently, $3A=G$ and $\ord^*_G(A) \leq 3$.

\noindent \textit{Case 2:} $0\in 3A$. Then $5A\supset2A$ and $6A\supset3A$. Since $5A\cap6A=\emptyset$, $5A\cap3A=\emptyset$. Thus, $3A\cup4A\cup5A=G$ is a partition of $G$. On the other hand, since $2A\cup3A\cup4A= G$, we deduce that $2A=5A$
Similarly to the previous case, we have $\ord^*_G(A) \leq 2$.

\noindent \textit{Case 3:} $0\in 4A.$ Then $6A\supset2A$, $7A\supset3A$, $8A\supset4A$, and $2A\cup3A\cup4A=G=6A\cup7A\cup8A=G$. Since $7A$ is disjoint from $6A$ and $8A$, we deduce $3A=7A$. 
Similarly to the previous case, we have $\ord^*_G(A) \leq 3$.
\end{proof}

\section{The function $S_G$} \label{sec:s}
The key in generalizing Cassaigne and Plagne's argument \cite{cp} is the notion of \textit{amenability}. Among the many equivalent definitions of amenability, we work with the one defined in terms of \textit{invariant means}. 
Let $G$ be a discrete (not necessarily abelian) group. Let $l^{\infty}(G)$ denote the set of all bounded functions on $G$. 
A right-invariant mean on $G$ is a linear functional $\Lambda: l^{\infty}(G) \rightarrow \R$ satisfying:
\begin{enumerate}[(i)]
 \item $\Lambda$ is nonnegative: if $f \geq 0$ on $G$, then $\Lambda(f) \geq 0$,
 \item $\Lambda$ has norm 1: $\Lambda(1_G)=1$ where $1_G$ is the characteristic function of $G$,
 \item $\Lambda$ is right-invariant: $\Lambda( \tau_g f) = \Lambda( f)$ for any $f \in l^{\infty}(G)$ and $g \in G$, where $\tau_g$ is the right translation: $\tau_g f(x) = f(xg)$.
\end{enumerate}
$G$ is called \textit{amenable} if there exists a right-invariant mean on $G$.

We recall here some standard facts about amenable groups. For a reference, see for example \cite[Appendix G]{bhv}. 
The additive group of the integers $\Z$ is amenable. The existence of invariant means on $\Z$ is non-constructive, since it requires either the use of ultrafilters or the Hahn-Banach theorem. More generally, any discrete abelian group is amenable. The free subgroups on two generators is not amenable.

\begin{proof}[Proof of Theorem \ref{th:main-s1}] The lower bound $h+1 \leq S_G(h)$ is an immediate consequence of Theorem \ref{th:main-minimal}. We will now prove the upper bound.
Let $\Lambda$ be an invariant mean on $G$. Since $G$ is infinite, it is easy to see that $\Lambda(1_I) = 0$ for all finite subset $I \subset G$, where $1_I$ is the characteristic function of $I$ (it suffices to see this for a singleton). 

Let $A$ be a basis of order $h$ of $G$. Without loss of generality, we may assume that $0 \in A$.

 For each element $a \in A$, let $f_a$ be the function on $G$ defined by
 \[
f_a(x)=
\left\{
  \begin{array}{ll}
    1, & \hbox{if } x \in hA \setminus h \left( A \setminus \{a\} \right) \\
    0, & \hbox{otherwise.}
  \end{array}
\right.
\]
In other words, $f_a(x)=1$ if and only if $a$ is essential in all representations of $x$ as a sum of $h$ elements from $A$. 

Just like in the proof given in \cite{cp}, we make two observations. First, for any $x \in G$ and finite subset $I \subset A$, we have $\sum_{a \in I} f_a(x) \leq h$. 
Indeed, if $x \not \in hA$ clearly then $f_a(x)=0$ for any $a \in A$.
Suppose $x \in hA$. Fix a representation
\[
 x = a_1 +\cdots + a_h
\]
where $a_i \in A$. Then $f_a(x)$ can only be 1 if $a$ is one of the $a_i$, and there are at most $h$ of these. 

Applying $\Lambda$ to both sides, we have proved the following statement.

\begin{claim} \label{claim:sum}
 For any finite subset $I \subset A$, we have 
\[
\sum_{a \in I} \Lambda (f_a) \leq h.  
\]
\end{claim}

Our next claim is this :

\begin{claim} \label{claim:rep}
If $a \in A$ is such that $\Lambda(f_a) < 1/h$, then there exists $x \in G$ such that
\[
  x + ia \in h (A \setminus \{ a\})
\]
for any $i=0, 1, \ldots, h-1$.
\end{claim}
Indeed, since $\Lambda$ is translation invariant, we have
\[
1 > h \Lambda(f_a) = \sum_{i=0}^{h-1} \Lambda \left( \tau_{ia} f_a \right) = \Lambda \left( \sum_{i=0}^{h-1} \tau_{ia}f_a \right).
\]
It follows that for an infinite set $B \subset G$, we have
\[
 1 > \sum_{i=0}^{h-1} \tau_{ia}f_a(x) = \sum_{i=0}^{h-1} f_a(x+ia).
\]
for all $x \in B$. Consequently, for any $x \in B$, $f_a(x+ia) = 0$ for any $i=0,1, \ldots, h-1$.

Since $hA \sim G$, there must exist such an $x \in B$ such that $x+ia \in hA$ for any $i=0,1, \ldots, h-1$. For this $x$, we have
\[
  x + ia \in h (A \setminus \{ a\})
\]
for any $i=0, 1, \ldots, h-1$, as required.

From Claim \ref{claim:sum} it follows that for all but finitely many $a \in A$, we have $a \neq 0$ and $\Lambda(f_a) < 1/h$. 
For such an $a$, let 
$x$ be such that $x + ia \in h (A \setminus \{ a\})$ for any $i=0, 1, \ldots, h-1$, whose existence is given by Claim \ref{claim:rep}. 
Now for all but finitely many $y \in G$, we have $y-x \in hA$ and $y-x \neq ha$. By removing any occurrence of $a$, it follows that for some $0 \leq i \leq h-1$, we have
\[
 y-x-ia \in (h-i) (A \setminus \{a\}).
\]
This implies that
\[
 y = (y-x-ia) + (x+ia) \in (2h-i) (A \setminus \{a\}) \subset 2h(A \setminus \{a\})
\]
which proves that $A \setminus \{a\}$ is a basis of order at most $2h$.
\end{proof}

Notice that the definition of amenability can be extended to locally compact groups. The definition is the same, except that we replace $l^\infty(G)$ by $L^\infty(G, \mu)$ where $\mu$ is a Haar measure on $G$. 
 Again, it is known that all abelian locally compact groups are amenable. The argument above can be applied to these groups as well, with an appropriate change in the definition of order. 
 Instead of requiring $hA \sim G$, we require $\mu(G \setminus hA) =0$. 

\begin{proof}[Proof of Theorem \ref{th:main-s2}] Since $S_G(2) \geq 3$, it suffices to show that $S_G(2) \leq 3$. Let $A$ be a basis of order at most $2$ of $G$. 
Call $b \in A$ \textit{bad} if $\ord^*(A \setminus \{ b\}) \geq 4$ and \textit{good} otherwise. We will show that $A$ has only finitely many bad elements.

By considering $A-c$ instead of $A$ where $c$ is any element of $A$, we may assume that $0 \in A$.

We first examine properties of a bad element $b \in A, b \neq 0$. Let us write $A_b=A \setminus \{b\}$. Since $A$ is a basis of order 2, we have
\begin{equation} \label{eq:u1}
 G \sim 2A \sim 2 A_b \cup \left( A_b + b \right).
\end{equation}
Let $a$ be an arbitrary element of $A_b$. Then
\begin{equation} \label{eq:u2}
 G \sim 2A + a \sim \left( 2A_b + a \right) \cup \left( A_b + b + a \right) \subset  3 A_b \cup \left( A_b + b + a\right).
\end{equation}
From (\ref{eq:u1}) we also deduce
\begin{equation} \label{eq:u3}
 G \sim 2 A_b \cup \left( A_b + b \right) \subset 3 A_b \cup (A_b + b).
\end{equation}
From (\ref{eq:u2}) and (\ref{eq:u3}) we see that the sets $(A_b + b + a)$ and $(A_b + b)$ both contain all but finitely many elements of $G \setminus 3A_b$. 
Since $b$ is bad, this implies that $(A_b + b + a) \cap (A_b + b)$ is infinite. 
In other words, we have proved:

\noindent \textbf{Claim 1:} For any $a \in A_b$, $(A_b + a) \cap A_b$ is infinite.

Next we prove
 
\noindent \textbf{Claim 2:} $(A_b + b) \cap A_b = \emptyset$.

Indeed, suppose for a contradiction that there are $a_1, a_2 \in A_b$ such that $b+a_1 =a_2$. For all but finitely many $x \in A$, we have $x-a_1 \in 2A \setminus \{2b\}$. 
If $x-a_1 \in 2 A_b$ then $x \in 3A_b$. If $x-a_1 \in A_b + b$ then $x \in A_b + a_2 \subset 2A_b \subset 3 A_b$. Thus $3A_b \sim G$, a contradiction.

Suppose now that there is another bad element $b' \in A, b' \neq 0$. From Claim 1 we know that $(A_b + b')\cap A_b$ is infinite. 
Therefore, $ \left( A \setminus \{b,b'\} + b' \right) \cap \left( A \setminus \{b,b'\} \right)$ is infinite. But this contradicts Claim 2 (with $b$ replaced by $b'$).
\end{proof}

In fact, the proof shows that there is at most one bad element in $A$. Indeed, the above argument shows that for any $c \in A$, there is at most one bad element in $A$ that is different from $c$. 
Applying this observation to a good element $c$ (which we know to exist), this implies that there is at most one bad element in $A$.

\section*{Acknowledgements}
The authors are supported by the ANR grant C\ae sar, number ANR 12-BS01-0011. The second author is supported by the Fondation Math\'ematique Jacques Hadamard. We would like to thank P. Longobardi and M. Maj for a useful discussion.

\end{document}